\documentclass[12pt]{article}
\usepackage[a4paper,margin=1.0in]{geometry}

\usepackage[colorlinks,citecolor=magenta,linkcolor=black]{hyperref}
\pdfpagewidth=\paperwidth \pdfpageheight=\paperheight
\usepackage{amsfonts,amssymb,amsthm,amsmath,eucal,tabu,url}
\usepackage{pgf}
 \usepackage{array}
 \usepackage{tikz-cd}
 \usepackage{pstricks}
 \usepackage{pstricks-add}
 \usepackage{pgf,tikz}
 \usetikzlibrary{automata}
 \usetikzlibrary{arrows}
 \usepackage{indentfirst}
\usepackage{tabularx} 

\theoremstyle{plain}
\newtheorem{thm}{Theorem}[section]
\newtheorem{theorem}[thm]{Theorem}

\newtheorem{proposition}[thm]{Proposition}

\newtheorem{conjecture}[thm]{Conjecture}

\theoremstyle{definition}
\newtheorem{definition}[thm]{Definition}
\newtheorem{remark}[thm]{Remark}
\newtheorem{example}[thm]{Example}

\newtheorem{thevarthm}[thm]{\varthmname}

\newenvironment{varthm*}[1]{\trivlist\item[]{\bf #1.}\it}{\endtrivlist}

\def\keywordname{{\bfseries Keywords}}%
\def\keywords#1{\par\addvspace\medskipamount{\rightskip=0pt plus1cm
\def\and{\ifhmode\unskip\nobreak\fi\ $\cdot$
}\noindent\keywordname\enspace\ignorespaces#1\par}}
\def\subclassname{{\bfseries Mathematics Subject Classification
(2020)}\enspace}
\def\subclass#1{\par\addvspace\medskipamount{\rightskip=0pt plus1cm
\def\and{\ifhmode\unskip\nobreak\fi\ $\cdot$
}\noindent\subclassname\ignorespaces#1\par}}

\begin{document}
\title{On Poincar\'e polynomials for plane curves with quasi-homogeneous singularities}
\author{Piotr Pokora}
\date{\today}
\maketitle

\thispagestyle{empty}
\begin{abstract}
We define a combinatorial object that can be associated with any conic-line arrangement with ordinary singularities, which we call the combinatorial Poincar\'e polynomial. We prove a Terao-type factorization statement on the splitting of such a polynomial over the rationals under the assumption that our conic-line arrangements are free and admit ordinary quasi-homogeneous singularities. Then we focus on the so-called $d$-arrangements in the plane. In particular, we provide a combinatorial constraint for free $d$-arrangements admitting ordinary quasi-homogeneous singularities.
\keywords{conic-line arrangements, freeness, singularities of plane curves, Poincar\'e polynomials}
\subclass{14H50, 14N20, 32S25}
\end{abstract}
\section{Introduction}
In the present paper we continue our studies on the freeness of plane curve arrangements with ordinary quasi-homogeneous singularities, which we started in \cite{PP2023, PP2024}. Let us briefly recall that a reduced plane curve is said to be {\it free} if its associated module of derivatives is a free module over the polynomial ring. Our investigations in this paper are divided into two parts. In the first part, we focus on conic-line arrangements with ordinary quasi-homogeneous singularities. Our motivation to study such arrangements follows from the so-called Numerical Terao's Conjecture, which tells us that the freeness of such conic-line arrangements is governed by the vector of weak combinatorics. More precisely, if $C \subset \mathbb{P}^{2}_{\mathbb{C}}$ is a reduced plane curve such that all irreducible components are smooth and $C$ has only ordinary singularities (i.e., they look locally like $x^{r} = y^{r}$ for some $r\geq 2$), then the weak-combinatorics is defined as the vector of the form
\[W(C) = \{k_{1}, \ldots , k_{l}; n_{2}, \ldots , n_{t}\},\]
where $k_{i}$ denotes the number of irreducible components of degree $i$ and $n_{j}$ denotes the number of $j$-fold intersections in $C$.
\begin{conjecture}[Numerical Terao's Conjecture -- special case]
Let $C_{1}, C_{2}$ be two reduced curves in $\mathbb{P}^{2}_{\mathbb{C}}$ such that their all irreducible components are smooth and the curves admit only ordinary quasi-homogeneous singularities. Assume that $C_{1}$ and $C_{2}$ have the same weak combinatorics, i.e., $W(C_{1}) = W(C_{2})$, and $C_{1}$ is free. Then $C_{2}$ has to be free.
\end{conjecture}
This conjecture is a very natural generalization of the classical Terao's conjecture on line arrangements in the complex projective plane, but in this situation we focus on the intersection posets of line arrangements as the decisive objects for the freeness. It is well-known that the above Numerical Terao's Conjecture is false, or more precisely, it is false in the class of \textit{triangular line arrangements} \cite{Mar}. However, except the case of line arrangements, \textit{we are not aware of any counterexample}. To approach this problem, it is natural to construct some combinatorial conditions/constraints that will allow us to search for good candidates to construct potential counterexamples to this conjecture. Let us recall that in the class of line arrangements we have a very good object that helps us to search for the free arrangements, namely the reduced Poincar\'e polynomial.
\begin{definition}
Let $\mathcal{L}\subset \mathbb{P}^{2}_{\mathbb{C}}$ be an arrangement of $d$ lines. Then the reduced Poincar\'e polynomial of $\mathcal{L}$ is defined as
$$\pi_{0}(\mathcal{L};t) = 1 + (d-1)t + \bigg(\sum_{r\geq 2}(r-1)n_{r}-d+1\bigg)t^{2}.$$
\end{definition}
 It is worth recalling there that the Poincar\'e polynomial allows us to compute, for instance, the Betti numbers of the complex complement of $\mathcal{L}$, i.e., the Betti numbers of $M=\mathbb{P}^{2}_{\mathbb{C}} \setminus \bigcup_{H \in \mathcal{L}}H$. Moreover, the following result is true.
\begin{theorem}[{Terao's factorization, \cite{terao}}]
Let $\mathcal{L} \subset \mathbb{P}^{2}_{\mathbb{C}}$ be an arrangement of $d$ lines. Suppose that $\mathcal{L}$ is free, then $\pi_{0}(\mathcal{L};t)$ splits over the rational numbers and we have the following presentation:
$$\pi_{0}(\mathcal{L};t) = (1+d_{1}t)(1+d_{2}t),$$
where $(d_{1}, d_{2})$ with $d_{1} \leq d_{2}$, and satisfying $d_{1} + d_{2} = d-1$, are the exponents of $\mathcal{L}$.
\end{theorem}
Now we would like to introduce a new combinatorial object for conic-line arrangements that will play a similar role as the reduced Poincar\'e polynomial for line arrangements.
\begin{definition}
Let $\mathcal{CL} \subset \mathbb{P}^{2}_{\mathbb{C}}$ be an arrangement of $d$ lines and $k$ smooth conics such that the arrangement admits only ordinary intersection points. Then the combinatorial Poincar\'e polynomial of $\mathcal{CL}$ is defined as
\begin{equation}
\mathfrak{P}(\mathcal{CL};t) = 1 + (2k+d-1)t + \bigg( \sum_{r\geq 2}(r-1)n_{r} -d+1 \bigg)t^{2}.
\end{equation}
\end{definition}
It is clear from the definition that if we have two conic-line arrangements with the same vectors of weak combinatorics, then their combinatorial Poincar\'e polynomials are equal. At first glance, the combinatorial Poincar\'e polynomial for our conic-line arrangements looks like a rabbit out of a hat, but it turns out this is going to be a crucial concept for our research. It is also worth noting that if $k=0$, then we get exactly the classical reduced Poincar\'e polynomial for line arrangements. However, as we might expect, our combinatorial Poincar\'e polynomial does not decode information about the Betti numbers of the complement $M=\mathbb{P}^{2}_{\mathbb{C}} \setminus \bigcup_{C \in \mathcal{CL}}C$. On the other hand, as a consolation prize, the combinatorial Poincar\'e polynomial allows us to compute the Euler number of the complement, and for details we now refer to Remark \ref{euler} below.
\newpage 
Our main result in the paper is the following Terao-type factorization statement.

\begin{theorem}
\label{main}
Let $\mathcal{CL} \subset \mathbb{P}^{2}_{\mathbb{C}}$ be an arrangement of $d$ lines and $k$ smooth conics such that the arrangement admits ordinary quasi-homogeneous intersection points. Assume that $\mathcal{CL}$ is free, then the combinatorial Poincar\'e polynomial splits over the rational numbers and then it has the following presentation
$$\mathfrak{P}(\mathcal{CL};t) = (1+d_{1}t)(1+d_{2}t),$$
where $(d_{1},d_{2})$ are the exponents of $\mathcal{CL}$.
\end{theorem}
Our result can be considered as a complementary tool that supports the investigations on conic-line arrangements presented in \cite{ST}.

The next goal is to explain how to construct meaningful combinatorial Poincar\'e polynomials when we allow to use non-ordinary singularities. Here we present our result devoted to a certain class of conic arrangements in the complex plane.

\begin{theorem}
\label{conics}
 Let $\mathcal{C} \subset \mathbb{P}^{2}_{\mathbb{C}}$ be an arrangement of $k\geq 2$ smooth conics admitting $n_{2}$ ordinary double points, $n_{3}$ ordinary triple points, $n_{4}$ ordinary quadruple points, $t_{3}$ singularities of type $A_{3}$, $t_{5}$ singularities of type $A_{5}$, and $t_{7}$ singularities of type $A_{7}$. If $\mathcal{C}$ is free with exponents $(d_{1},d_{2})$, then its combinatorial Poincar\'e polynomial defined as 
 $$\mathfrak{P}(\mathcal{C};t) = 1 + (2k-1)t + \bigg(n_{2} + 2n_{3} + 3n_{4} +t_{3} + t_{5} + t_{7} + 1\bigg)t^{2}$$
 splits over the rationals, and we have $\mathfrak{P}(\mathcal{C};t) = (1+d_{1}t)(1+d_{2}t).$
\end{theorem}

Then we focus on the freeness of $d$-arrangements in the complex projective plane.
\begin{definition}
We say that an arrangement of plane curves $\mathcal{C} = \{C_{1}, ...,C_{k}\}\subset \mathbb{P}^{2}_{\mathbb{C}}$ is called as a $d$-arrangement if
\begin{itemize}
\item all curves $C_{i}$ are smooth of the same degree $d\geq 1$,
\item all intersection points of $\mathcal{C}$ are ordinary singularities.
\end{itemize}
\end{definition}
Our main contribution in this setting is the following result which can be seen as a generalization of \cite[Theorem 5.2]{PP2023}.
\begin{theorem}
\label{dd}
Let $\mathcal{C}$ be a $d$-arrangement with $k \geq 2$ and $d\geq 2$. Assume that all singularities of $\mathcal{C}$ are quasi-homogeneous and $\mathcal{C}$ is free. Then the following inequality holds:
\begin{equation}
    \sum_{r\geq 2}(r^{2}-5r+4) n_{r} \geq 3 + 3dk(d-2)
\end{equation}
\end{theorem}

Let us present an outline of our paper. In Section $2$ we present all necessary preliminaries devoted to free plane curves admitting quasi-homogeneous singularities and then, in Section $3$, we present our proofs of Theorems \ref{main}, \ref{conics}, and \ref{dd}.

We work exclusively over the complex numbers.
\section{Basics on free plane curves}

We follow the notation introduced in \cite{Dimca}. Let us denote by $S := \mathbb{C}[x,y,z]$ the coordinate ring of $\mathbb{P}^{2}_{\mathbb{C}}$. For a homogeneous polynomial $f \in S$ let $J_{f}$ denote the Jacobian ideal associated with $f$, i.e., the ideal of the form $J_{f} = \langle \partial_{x}\, f, \partial_{y} \, f, \partial_{z} \, f \rangle$.
\begin{definition}
Let $p$ be an isolated singularity of a polynomial $f\in \mathbb{C}[x,y]$. Since we can change the local coordinates, assume that $p=(0,0)$.
\begin{itemize}
    
\item The number 
$$\mu_{p}=\dim_\mathbb{C}\left(\mathbb{C}\{x,y\} /\bigg\langle \partial_{x}\, f ,\partial_{y}\, f \bigg\rangle\right)$$
is called the Milnor number of $f$ at $p$.

\item The number
$$\tau_{p}=\dim_\mathbb{C}\left(\mathbb{C}\{x,y\}/\bigg\langle f, \partial_{x}\, f ,\partial_{y} \, f \bigg\rangle \right)$$
is called the Tjurina number of $f$ at $p$.

\end{itemize}
\end{definition}

The total Tjurina number of a given reduced curve $C \subset \mathbb{P}^{2}_{\mathbb{C}}$ is defined as
$${\rm deg}(J_{f}) = \tau(C) = \sum_{p \in {\rm Sing}(C)} \tau_{p}.$$ 

Recall that a singularity is called quasi-homogeneous if and only if there exists a holomorphic change of variables so that the defining equation becomes weighted homogeneous. If $C : f=0$ is a reduced plane curve with only quasi-homogeneous singularities,  then by \cite[Satz]{KS} one has $\tau_{p}=\mu_{p}$ for all $p \in {\rm Sing}(C)$, and eventually
$$\tau(C) = \sum_{p \in {\rm Sing}(C)} \tau_{p} = \sum_{p \in {\rm Sing}(C)} \mu_{p} = \mu(C),$$
which means that the total Tjurina number is equal to the total Milnor number of $C$.

Next, we will need an important invariant that is defined in the language of the syzygies of $J_{f}$.
\begin{definition}
Consider the graded $S$-module of Jacobian syzygies of $f$, namely $$AR(f)=\{(a,b,c)\in S^3 : a\partial_{x} \, f + b \partial_{y} \, f + c \partial_{z} \, f = 0 \}.$$
The minimal degree of non-trivial Jacobian relations for $f$ is defined to be 
$${\rm mdr}(f):=\min_{r\geq 0}\{AR(f)_r\neq 0\}.$$ 
\end{definition}
\begin{remark}
If $C: f=0$ is a reduced plane curve in $\mathbb{P}^{2}_{\mathbb{C}}$, then we write ${\rm mdr}(f)$ or ${\rm mdr}(C)$ interchangeably.
\end{remark}
Let us now formally define the freeness of a reduced plane curve \cite{KS1}.
\begin{definition}
A reduced curve $C \subset \mathbb{P}^{2}_{\mathbb{C}}$ of degree $d$ is free if the Jacobian ideal $J_{f}$ is saturated with respect to $\mathfrak{m} = \langle x,y,z\rangle$. Moreover, if $C$ is free, then the pair $(d_{1}, d_{2}) = ({\rm mdr}(f), d - 1 - {\rm mdr}(f))$ is called the exponents of $C$.
\end{definition}
It is difficult to check the freeness property using the above definition. However, it turns out that we can use the following result due to du Plessis and Wall \cite{duP}.
\begin{theorem}
\label{dddp}
Let $C : f=0$ be a reduced curve in $\mathbb{P}^{2}_{\mathbb{C}}$. One has
\begin{equation}
\label{duPles}
(d-1)^{2} - d_{1}(d-d_{1}-1) = \tau(C)
\end{equation}
if and only if $C : f=0 $ is a free curve, and then ${\rm mdr}(f) = d_{1} \leq (d-1)/2$.
\end{theorem}
\section{Combinatorial Poincar\'e polynomials versus freeness}
We start this section by the following.
\begin{proposition}
\label{ccc}
Let $\mathcal{CL}$ be an arrangement of $d$ lines and $k$ smooth conics in the complex projective plane. Assume that $\mathcal{CL}$ admits only ordinary quasi-homogeneous singularites and let $\mathcal{CL}$ be free with exponents $(d_{1}, d_{2})$. Then
\begin{equation}
\sum_{r\geq 2}(r-1)n_{r} - d + 1 = d_{1}d_{2}.
\end{equation}
\end{proposition} 
\begin{proof}
Since $\mathcal{CL} : f = 0$ is free, by Theorem \ref{dddp} we have 
\begin{equation}
\label{cc}
\tau(\mathcal{CL}) = (2k+d-1)^{2} - d_{1}(2k+d-d_{1}-1),
\end{equation}
where $d_{1} = {\rm mdr}(f)$. Moreover, since $2k+d - 1 = d_{1}+d_{2}$, then we have
$$\tau(\mathcal{CL}) = (2k+d-1)^{2} - d_{1}(2k+d-d_{1}-1) = (2k+d-1)^{2} - d_{1}(d_{1}+d_{2}-d_{1}) = (2k+d-1)^{2} - d_{1}d_{2}.$$

Now we need to compute $\tau(\mathcal{CL})$. Since our singularities are ordinary and quasi-homogeneous, then
$$\tau(\mathcal{CL}) = \mu(\mathcal{CL}) = \sum_{p \in {\rm Sing}(\mathcal{CL})}({\rm mult}_{p}(\mathcal{CL})-1)^{2} = \sum_{r\geq 2}(r-1)^{2}n_{r}.$$
Let us denote by $f_{i} = \sum_{r\geq 2}r^{i}n_{r}$, then
$$\tau(\mathcal{CL}) = f_{2}-2f_{1} + f_{0}.$$
Observe that the following naive combinatorial count holds (in fact, this count holds regardless of whether the singularities are quasi-homogeneous or not):
$$4\binom{k}{2} + 2kd + \binom{d}{2} = \sum_{r\geq 2}\binom{r}{2}n_{r},$$
and we can write it as 
$$4k^{2}-4k + 4kd + d^{2} - d = f_{2}-f_{1}.$$
Moreover, since 
$$4k^{2}-4k + 4kd + d^{2} - d =(2k+d-1)^{2} + d - 1,$$
we finally get
$$(2k+d-1)^{2} + d - 1 = f_{2} - f_{1}.$$
Let us come back to $\tau(\mathcal{CL})$, we have
$$\tau(\mathcal{CL}) = f_{2} - f_{1} + (-f_{1}+f_{0}) = (2k+d-1)^{2} + d - 1 - f_{1}+f_{0}.$$
Plugging this identity to \eqref{cc}, we arrive at
$$(2k+d-1)^{2} + d - 1 - f_{1}+f_{0} = (2k+d-1)^{2} - d_{1}d_{2},$$
so finally
$$f_{1} - f_{0} - d + 1 = d_{1}d_{2},$$
which completes the proof.
\end{proof}
Now we are ready to give our proof of Theorem \ref{main}.
\begin{proof}
By our assumptions, $\mathcal{CL}$ is a free arrangement of $d$ lines and $k$ smooth conics with only ordinary quasi-homogeneous singularities, hence $d_{1} + d_{2} = 2k+d-1$, and by Proposition \ref{ccc} one has 
$$\sum_{r\geq 2}(r-1)n_{r} -d+1 =d_{1}d_{2},$$
where $(d_{1},d_{2})$ are the exponents of $\mathcal{CL}$. This gives us 
\begin{multline*}
\mathfrak{P}(\mathcal{CL};t) = 1 + (2k+d-1)t + \bigg( \sum_{r\geq 2}(r-1)n_{r} -d+1 \bigg)t^{2} = \\ 1 + (d_{1}+d_{2})t + d_{1}d_{2}t^{2} = (1+d_{1}t)(1+d_{2}t),    
\end{multline*}
which completes then proof.
\end{proof}

Let us now present how Theorem \ref{main} works in practice.
\begin{example}
Consider the conic-line arrangement $\mathcal{CL} \subset \mathbb{P}^{2}_{\mathbb{C}}$ defined as follows:
\begin{equation*}
\mathcal{CL} \, : (-24x^{2}-23y^{2}+76yz+195z^{2})(y-3x-5z)(y+3x-5z)(y+z)(y-3z)x(x+y+z) = 0.
\end{equation*}
The weak-combinatorics of $\mathcal{CL}$ has the following form
$$W(\mathcal{CL}) = (k_{1}, k_{2}; n_{2}, n_{3}, n_{4}) = (6,1;12,3,1).$$
The combinatorial Poincar\'e polynomial of $\mathcal{CL}$ has the form
$$\mathfrak{P}(\mathcal{CL};t) = 1+7t+16t^2$$
and it does not split over the rationals, so $\mathcal{CL}$ cannot be free.
\end{example}
\begin{example}
This example comes from \cite{PP2024}. Let us consider the following conic-line arrangement defined as follows:
\begin{align*}
\mathcal{CL}  : xy(x-z)(x+z)(y-z)(y+z)(y-x-z)&(y-x+z)(y-x)\\&(-x^2+xy-y^2+z^2)=0.      
\end{align*}
The weak-combinatorics of $\mathcal{CL}$ has the following form
$$W(\mathcal{CL}) = (k_{1}, k_{2}; n_{2}, n_{3}, n_{4}) = (9, 1; 6, 4,6).$$
Recall that by \cite[Theorem 1.3]{PP2024} the arrangement $\mathcal{CL}$ is free with exponents $(4,6)$ -- this can be also checked directly by using Theorem \ref{dddp} and the fact that $\tau(\mathcal{CL})=76$. We compute the combinatorial Poincar\'e polynomial of $\mathcal{CL}$, namely
$$\mathfrak{P}(\mathcal{CL};t) = 1+10t+24t^2 = (1+4t)(1+6t),$$
hence the splitting over the rationals holds. 
\end{example}
\begin{example}
Consider the following line arrangement $\mathcal{L} \subset \mathbb{P}^{2}_{\mathbb{C}}$ given by
$$Q(x,y,z) = xy(x+y+z)(x+y+2z)(x+3y+z)(5x+y+z).$$
Note that in this case we have $W(\mathcal{L}) = (k_{1}; n_{2},n_{3}) = (6;9,2)$ and 
$$\pi_{0}(\mathcal{L};t) = \mathfrak{P}(\mathcal{L};t) = 1+5t + 8t^{2}.$$
We see that $\pi_{0}(\mathcal{L};t)$ does not split over the rationals, so $\mathcal{L}$ cannot be free. However, this arrangement is \textit{plus-one generated} with exponents $(3,3,4)$ -- see \cite{ABE} for necessary details regarding plus-one generated arrangements.
\end{example}
\begin{remark}
\label{euler}
Let us consider the complement $M=\mathbb{P}^{2}_{\mathbb{C}} \setminus \bigcup_{C \in \mathcal{CL}}C$, where $\mathcal{CL}\subset \mathbb{P}^{2}_{\mathbb{C}}$ is an arrangement of $d$ lines and $k$ smooth conics admitting only ordinary singularities. Recall that the Betti polynomial of $M$ has the form
$$B_{M}(t) = 1 + (k+d-1)t+ \bigg(\sum_{r\geq 2}(r-1)n_{r}-d-k+1 \bigg)t^{2},$$
see \cite[Section 2.1]{Cog} and \cite[Section 1.2]{ST} for further explanations.
Observe that the following equality holds
\begin{equation}
    \mathfrak{P}(\mathcal{CL};t) = B_{M}(t) + kt(t+1).
\end{equation}
Then the Euler number of $M$ is equal to
$$e(M) = B_{M}(-1) = \mathfrak{P}(\mathcal{CL};-1),$$
which means, somehow surprisingly, that the Euler number of $M$ can be computed directly via $\mathfrak{P}(\mathcal{CL};t)$.
\end{remark}
As we have mentioned in Introduction, our idea of the combinatorial Poincar\'e polynomial can be extended to the case of non-ordinary quasi-homogeneous singularities. As an exemplary result, we focus on conic arrangements in the plane admitting some non-ordinary singularities. We use the notation of local normal forms of plane curve singularties from \cite{arnold}.
\begin{definition}
Let $\mathcal{C} \subset \mathbb{P}^{2}_{\mathbb{C}}$ be an arrangement of $k\geq 2$ smooth conics admitting $n_{2}$ ordinary double points, $n_{3}$ ordinary triple points, $n_{4}$ ordinary quadruple points, $t_{3}$ singularities of type $A_{3}$, $t_{5}$ singularities of type $A_{5}$, and $t_{7}$ singularities of type $A_{7}$. Then the Poincar\'e polynomial of $\mathcal{C}$ is defined as
$$\mathfrak{P}(\mathcal{C};t) = 1 + (2k-1)t + \bigg(n_{2} + 2n_{3} + 3n_{4} +t_{3} + t_{5} + t_{7} + 1\bigg)t^{2}.$$
\end{definition}
This definition follows from the following observation, which also gives us a direct proof of Theorem \ref{conics}.
\begin{proposition}
Let $\mathcal{C} \subset \mathbb{P}^{2}_{\mathbb{C}}$ be an arrangement of $k\geq 2$ smooth conics admitting $n_{2}$ ordinary double points, $n_{3}$ ordinary triple points, $n_{4}$ ordinary quadruple points, $t_{3}$ singularities of type $A_{3}$, $t_{5}$ singularities of type $A_{5}$, and $t_{7}$ singularities of type $A_{7}$. Assume that $\mathcal{C}$ is free with exponents $(d_{1},d_{2})$, then
\begin{equation}
\label{con}
n_{2}+2n_{3}+3n_{4} + t_{3}+t_{5}+t_{7} + 1 = d_{1}d_{2}.
\end{equation}
\end{proposition}
\begin{proof}
Our proof here goes along the same lines as of Theorem \ref{main}, so let us present only a sketch of the proof.
The freeness of $\mathcal{C}$ implies that
$$\tau(\mathcal{C}) = (2k-1)^{2} - d_{1}d_{2}.$$
Now we need to recall that in this setting one has
$$\tau(C) = n_{2} + 4n_{3} + 9n_{4} + 3t_{3} + 5t_{5} + 7t_{7}.$$
Moreover, the following combinatorial count holds:
$$4 \binom{k}{2} = n_{2} +3n_{3} + 6n_{4} + 2t_{3} + 3t_{5} + 4t_{7}.$$
Let us rewrite this count in a slightly different form, namely
$$(2k-1)^2 - 1 - n_{2} - 2n_{3} - 3n_{4} -t_{3} - t_{5} - t_{7} = \tau(\mathcal{C}),$$
and now, after combining the data collected above, we get the desired identity.
\end{proof}
Let us now use Theorem \ref{conics} in practice.
\begin{example}
It is well-know that there exists an arrangement $\mathcal{C}$ of $k=4$ smooth conics admitting exactly twelve singularities of type $A_{3}$ given by the following defining polynomial
$$Q(x,y,z) = (xy-z^{2})(xy+z^{2})(x^{2}+y^{2}-2z^{2})(x^{2}+y^{2}+2z^{2}),$$
see for instance \cite[Example 4.4]{Malara}. We can compute its Poincar\'e polynomial, namely
$$\mathfrak{P}(\mathcal{C};t) = 1+7t + 13t^2,$$
and it does not split over the rationals. In fact, we know that $\mathcal{C}$ is not free, but only \textit{nearly-free} with exponents $(4,4,4)$.
\end{example}
\begin{example}
This construction is presented in \cite[Remark 2.5]{PP2023}. There exists a unique up to projective equivalence arrangement $\mathcal{C}$ of $k=3$ conics having exactly one ordinary triple point and three singularities of type $A_{5}$. We know that $\mathcal{C}$ is free. One can compute the combinatorial Poincar\'e polynomial of $\mathcal{C}$, namely
$$\mathfrak{P}(\mathcal{C};t) = 1 + 5t + 6t^2 = (1+2t)(1+3t),$$
hence the splitting over the rationals holds.
\end{example}
\begin{remark}
One may ask if we can define the Poincar\'e polynomial for any reduced plane curve $C \subset \mathbb{P}^{2}_{\mathbb{C}}$ without any assumption about the singularities, and the answer is yes, but the resulting polynomial, in all its generality, cannot be of combinatorial nature, as we will see right now. More precisely, for a reduced plane curve $C \subset \mathbb{P}^{2}_{\mathbb{C}}$ of degree $d$ we define its Poincar\'e polynomial as
\begin{equation}
\mathfrak{P}(C;t) =1 + (d-1)t + ((d-1)^{2} - \tau(C))t^{2}.    
\end{equation}
One can show that $\mathfrak{P}(C;t)$ splits over the rationals provided that $C$ is free, and for more details regarding this subject we refer to \cite{DIMCC}. 
\end{remark}

Let us now focus on the freeness of $d$-arrangements and present our proof of Theorem \ref{dd}.
\begin{proof}
Keeping the same notation as in our proof of Theorem \ref{main}, the freeness of $\mathcal{C}$ implies that
$$\tau(\mathcal{C}) = f_{2}-2f_{1}+f_{0} = (dk-1)^{2} + d_{1}^{2} -d_{1}(dk-1),$$
where $d_{1} = {\rm mdr}(\mathcal{C})$.
Using the combinatorial count
$$d^{2}(k^{2}-k) = f_{2} - f_{1}$$
we get
$$d^{2}(k^{2}-k) - f_{1} + f_{0} = \tau(\mathcal{C}),$$
and hence
\begin{equation}
\label{dis}
d_{1}^{2} - d_{1}(dk-1) +d^{2}k - 2dk+1  + f_{1} - f_{0} = 0.
\end{equation}
Now we compute the discriminant $\triangle_{d_{1}}$ of \eqref{dis}, namely
$$\triangle_{d_{1}} = (dk-1)^2 - 4(d^{2}k-2dk+1 + f_{1}-f_{0}),$$
which by the freeness property it satisfies $\triangle_{d_{1}} \geq 0$.
This gives us
$$\triangle_{d_{1}} = d^{2}k^{2}-2dk + 1 - 4d^{2}k + 8dk - 4 -4f_{1}+4f_{0} = f_{2} - f_{1} -3d^{2}k+6dk -3 -4f_{1}+4f_{0} \geq 0.$$
After rearranging we get
$$f_{2} -5f_{1} + 4f_{0} = \sum_{r\geq 2}(r^{2}-5r + 4)n_{r} \geq 3 + 3dk(d-2),$$
which completes the proof.
\end{proof}
\begin{remark}
In the case of a $2$-arrangement $\mathcal{C}$ with quasi-ordinary singularities (i.e., conic arrangements with ordinary quasi-homogeneous singularities), the above result tells us that if $\mathcal{C}$ is free, then
$$\sum_{r\geq 5} (r^{2}-5r+4)n_{r} \geq 3 + 2n_{2} + 2n_{3}.$$
\end{remark}

\section*{Acknowledgments}
I would like to thank Alex Dimca for all useful remarks regarding the content of the paper and for suggesting Remark \ref{euler}. I would also like to thank an anonymous referee for useful suggestions.

Piotr Pokora is supported by the National Science Centre (Poland) Sonata Bis Grant  \textbf{2023/50/E/ST1/00025}. For the purpose of Open Access, the author has applied a CC-BY public copyright license to any Author Accepted Manuscript (AAM) version arising from this submission.

\vskip 0.5 cm
\bigskip
Piotr Pokora
Department of Mathematics,
University of the National Education Commission Krakow,
Podchor\c a\.zych 2,
PL-30-084 Krak\'ow, Poland. \\
\nopagebreak
\textit{E-mail address:} \texttt{piotr.pokora@uken.krakow.pl}
\end{document}